\documentclass[12pt, reqno]{amsart}
\usepackage{amsmath}
\usepackage{amssymb}
\usepackage{amsfonts}
\usepackage{amsthm}
\usepackage{color}
\usepackage{hyperref}
\usepackage{graphicx}
\usepackage{amsaddr}

\newtheorem{theorem}{Theorem}[section]

\newtheorem{lemma}[theorem]{Lemma}

\def\F{\mathbb{F}}

\begin{document}

\title{Three-dimensional Hadamard matrices of Paley type}

\author[V.~Kr\v{c}adinac, M.~O.~Pav\v{c}evi\'{c}, and K.~Tabak]{Vedran Kr\v{c}adinac$^1$, Mario Osvin Pav\v{c}evi\'{c}$^2$, and Kristijan Tabak$^3$}

\address{$^1$Faculty of Science, University of Zagreb, Bijeni\v{c}ka cesta~$30$, HR-$10000$ Zagreb, Croatia}

\address{$^2$Faculty of Electrical Engineering and Computing, University of Zagreb,
Unska~$3$, HR-$10000$ Zagreb, Croatia}

\address{$^3$Rochester Institute of Technology, Zagreb Campus,
D.~T.~Gavrana~$15$, HR-$10000$ Zagreb, Croatia}

\email{vedran.krcadinac@math.hr}
\email{mario.pavcevic@fer.hr}
\email{kristijan.tabak@croatia.rit.edu}

\thanks{This work has been supported by the Croatian Science Foundation
under the project $9752$.}



\date{September 12, 2023}

\begin{abstract}
We describe a construction of three-dimensional Hada\-mard matrices of
even order $v$ such that $v-1$ is a prime power. The construction covers
infinitely many orders for which the existence was previously open.
\end{abstract}

\maketitle

\section{Introduction}

An $n$-dimensional matrix of order $v$ over the set $S$ is a function
$H:\{1,\ldots,v\}^n\to S$. A $k$-dimensional \emph{layer} of $H$ is a restriction
obtained by fixing $n-k$ coordinates. Two layers are \emph{parallel} if the same
coordinates are fixed and the values of the fixed coordinates agree, except
(possibly) one. An \emph{$n$-dimensional Hadamard matrix} is an matrix over
$\{-1,1\}$ such that all $(n-1)$-dimensional parallel layers are mutually
orthogonal, i.e.
$$\sum_{1\le i_1,\ldots,\widehat{i_j},\ldots,i_n\le v} H(i_1,\ldots,a,\ldots,i_n)
H(i_1,\ldots,b,\ldots,i_n) = v^{n-1} \delta_{ab}$$
holds for all $j\in \{1,\ldots,n\}$ and $a,b\in \{1,\ldots,v\}$.
\emph{Proper $n$-dimensional Hadamard matrices} satisfy the stronger condition
that all $2$-dimensional layers are Hadamard, i.e.\ $\{-1,1\}$ matrices with
orthogonal rows and columns.

Higher-dimensional Hadamard matrices were introduced by Paul J.\ Schlichta \cite{PS71, PS79}.
These matrices have been extensively studied and have applications in signal processing,
error correction coding and cryptography; see~\cite{KH07, YNX10}.
A central question is determining the orders~$v$ for which Hadamard matrices
exist. By the famous Hadamard conjecture, $2$-dimensional matrices exist for
all~$v$ divisible by~$4$. The conjecture remains open with the smallest
order for which no example is known currently being $v=668$. By the following
theorem of Yang~\cite{YY86}, the existence of proper $n$-dimensional Hadamard
matrices is equivalent to the $2$-dimensional case.

\begin{theorem}[Product construction]
If $h=[h_{ij}]$ is a $2$-dimensional Hadamard matrix of order~$v$, then
$$H(i_1,\ldots,i_n)=\prod_{1\le j<k\le n} h_{i_j i_k}$$
is a proper $n$-dimensional Hadamard matrix of order~$v$.
\end{theorem}

On the other hand, orders of improper higher-dimensional Hada\-mard
matrices are not necessarily divisible by~$4$, but only even. The book \cite{YNX10}
contains many constructions for Hadamard matrices of orders $v\equiv 2 \pmod{4}$
and dimensions $n\ge 4$. A question whether such matrices exist for all
even orders is raised \cite[Question~12, p.~419]{YNX10}. There are
theorems giving higher-dimensional Hadamard matrices from lower-dimensional
ones, for example \cite[Theorem~6.1.5]{YNX10}:

\begin{theorem}\label{dimpp}
If $h:\{1,\ldots,v\}^n\to\{-1,1\}$ is an $n$-dimensional Hadamard matrix, then
$$H(i_1,\ldots,i_n,i_{n+1})=h(i_1,\ldots,i_{n-1},i_n+i_{n+1})$$
is an $(n+1)$-dimensional Hadamard matrix of the same order~$v$.
The sum in the last coordinate is taken modulo~$v$.
\end{theorem}

This construction does not preserve propriety, i.e.\ $H$ needs not be proper
even if~$h$ is proper. Less is known about Hadamard matrices of dimension~$3$.
A construction based on perfect binary arrays \cite[Theorem~3.2.2]{YNX10} gives
examples for orders $v=2\cdot 3^k$, $k\ge 0$. The existence of $3$-dimensional
Hadamard matrices of orders $v\equiv 2 \pmod{4}$ that are not of this form has
been a long-standing open problem \cite[Questions~5 and~6, p.~419]{YNX10}.
We give an affirmative answer for all even orders $v$ such that $v-1$ is a
prime power.

In the next section we present a construction based on squares in finite fields
of odd orders~$q$ similar to Paley's construction of $2$-dimensional Hadamard
matrices~\cite{RP33}. Our three-dimensional matrices are indexed by points of
the projective line $PG(1,q)$ and are invariant under the projective special
linear group $PSL(2,q)$. The construction covers infinitely many orders for
which three-dimensional Hadamard matrices were not known. We conclude the paper
in Section~\ref{sec3} by considering the smallest order not covered by our
construction that is still open.

\section{The construction}\label{sec2}

Shlichta suggested to generalise known algebraic constructions of
Hadamard matrices to higher dimensions \cite[Section VI, Problem~(a)]{PS79}.
A famous early construction using finite fields is due to Paley~\cite{RP33}.
It can be described as follows. Let $q$ be an odd prime power and
$\chi:\F_q^*\to \{1,-1\}$ a function that takes the value $1$ for non-zero
squares in the field~$\F_q$ and $-1$ for non-squares. Note that $\chi$
is a homomorphism from the multiplicative group of $\F_q$ to $\{1,-1\}$.
The rows and columns of the matrix~$h$ are indexed by the projective line
$PG(1,q)=\{\infty\}\cup \F_q$ and the elements are
\begin{equation}\label{paley}
h(x,y)=\left\{\begin{array}{ll}
-1, & \mbox{if } x=y=\infty,\\
1, & \mbox{if } x=y\neq \infty \mbox{ or } x=\infty\neq y \\
 & \mbox{ or } y=\infty\neq x,\\
\chi(y-x), & \mbox{otherwise.}
\end{array}\right.
\end{equation}
If $q\equiv 3 \pmod{4}$, this is a Hadamard matrix of order $q+1$ called
the \emph{Paley type~I} matrix. For $q\equiv 1 \pmod{4}$ there is a similar
construction of \emph{Paley type~II} Hadamard matrices of orders $2(q+1)$.

Hammer and Seberry~\cite[Example~2]{HS81} assume $\chi(0)=-1$ and
define the $n$-dimensional \emph{Paley cube} as
$$H(x_1,\ldots,x_n)=\left\{\begin{array}{ll}
1, & \mbox{if } x_i=\infty \mbox{ for at least one } i,\\
\chi(x_1+\ldots+x_n), & \mbox{otherwise.}
\end{array}\right.$$
For $n=2$ this matrix is equivalent to the Paley type~I matrix, but for
$n\ge 3$ the $(n-1)$-dimensional layers are not orthogonal and it is not
a higher-dimensional Hadamard matrix. Hammer and Seberry call it ``almost
Hadamard'' because the $2$-dimensional layers are either Hadamard matrices
or matrices of all ones.

We propose an alternative definition for $n=3$:
\begin{equation}\label{maindef}
H(x,y,z)=\left\{\begin{array}{ll}
-1, & \mbox{if } x=y=z,\\
1, & \mbox{if } x=y\neq z\\
 & \mbox{ or } x=z\neq y\\
 & \mbox{ or } y=z\neq x,\\
\chi(z-y), & \mbox{if } x=\infty,\\
\chi(x-z), & \mbox{if } y=\infty,\\
\chi(y-x), & \mbox{if } z=\infty,\\
\chi((x-y)(y-z)(z-x)), & \mbox{otherwise.}
\end{array}\right.
\end{equation}
In the last four rows the coordinates are assumed to be all distinct. We first
describe the symmetries of this three-dimensional matrix.

\begin{lemma}\label{lm1}
The matrix~\eqref{maindef} is invariant under cyclic shifts of the coordinates,
i.e.\ $H(x,y,z)=H(y,z,x)=H(z,x,y)$.
\end{lemma}

\begin{proof}
Follows directly from the definition.
\end{proof}

\begin{lemma}\label{lm2}
The matrix~\eqref{maindef} is invariant under linear fractional transformations
of the coordinates with determinant~$1$, i.e.\ under the action of $PSL(2,q)$ on
the projective line.
\end{lemma}

\begin{proof}
Let $f:PG(1,q)\to PG(1,q)$ be defined by $f(x)=\frac{ax+b}{cx+d}$ for $a,b,c,d\in \F_q$,
$ad-bc=1$. If the denominator is zero then $f(x)=\infty$, and $f(\infty)=\frac{a}{c}$.
We claim that $H(f(x),f(y),f(z))=H(x,y,z)$ for all $x,y,z\in PG(1,q)$. Since $f$ is a
bijection, this clearly holds if $x=y=z$ or two of the coordinates are equal. Assume
that the three coordinates are all distinct and $\infty$ does not appear among $x$, $y$, $z$,
$f(x)$, $f(y)$, $f(z)$. Because $ad-bc=1$, we have $f(x)-f(y)=\frac{x-y}{(cx+d)(cy+d)}$,
so $(f(x)-f(y))(f(y)-f(z))(f(z)-f(x))=\frac{(x-y)(y-z)(z-x)}{(cx+d)^2(cy-d)^2(cz+d)^2}$.
The denominator is a square, hence the $\chi$-value of this expression agrees with
$\chi((x-y)(y-z)(z-x))$. Similarly one can check that $H(f(x),f(y),f(z))=H(x,y,z)$
if $x$, $y$, $z$ are distinct and $\infty$ does appear among $x$, $y$, $z$,
$f(x)$, $f(y)$, $f(z)$.
\end{proof}

\begin{theorem}\label{maintm}
For every odd prime power~$q$, equation~\eqref{maindef} defines a three-dimensional Hadamard
matrix of order~$q+1$. If $q\equiv 3 \pmod{4}$, the matrix is proper with all $2$-dimensional
layers equivalent to the Paley type I matrix~\eqref{paley}.
\end{theorem}

\begin{proof}
Because of Lemma~\ref{lm1} we may fix the last coordinate. We claim that
for all distinct $a,b\in PG(1,q)$,
$$\sum_{x,y\in PG(1,q)} H(x,y,a) H(x,y,b)=0.$$
Now, because $PSL(2,q)$ acts $2$-transitively on $PG(1,q)$ and Lemma~\ref{lm2},
we may take $a=0$ and $b=\infty$ without loss of generality. We divide the sum
into two parts:
$$\sum_{x} H(x,x,0)H(x,x,\infty) + \sum_{x\neq y} H(x,y,0)H(x,y,\infty).$$
The first part is readily seen to be $q-3$. For the second part we distinguish whether
$0$, $\infty$ are among $x$, $y$ or not. If $x=0$, $y=\infty$ or $x=\infty$, $y=0$
the summand is $1$, so up to now the total is $q-1$. If $x\in \F_q^*$, $y=\infty$, the
sum is $$\sum_{x\in \F_q^*} H(x,\infty,0)H(x,\infty,\infty)=\sum_{x\in \F_q^*} \chi(x)=0$$
because there are equally many squares and non-squares in $\F_q^*$. Similarly we see that
the cases $x\in \F_q^*$, $y=0$; $x=\infty$, $y \in \F_q^*$; and $x=0$, $y \in \F_q^*$ sum
up to $0$. The final part of the sum is over $x,y\in \F_q^*$, $x\neq y$:
\begin{multline*}
\sum H(x,y,0)H(x,y,\infty)=\sum \chi((x-y)(y-0)(0-x))\chi(y-x)= \\[2mm]
= \sum \chi(x)\chi(y) =\sum_x \chi(x)\sum_{y\neq x}\chi(y) = \sum_x \chi(x)(-\chi(x))=1-q.
\end{multline*}
The grand total is $q-1+1-q=0$ and the first part of the theorem is proved. For the second
part notice that thanks to Lemmas~\ref{lm1} and~\ref{lm2} we may, without loss of generality,
fix the third coordinate to $z=\infty$ and look at the $2$-dimensional layer obtained by
varying $x$ and $y$. In this case equation~\eqref{maindef} reduces to~\eqref{paley}, and
this is a Paley type~I Hadamard matrix if $q\equiv 3 \pmod{4}$.
\end{proof}

\section{Concluding remarks}\label{sec3}

Theorem~\ref{maintm} proves the existence of three-dimensional Hadamard matrices
of orders $v=10$, $14$, $26$, $30$, $38$, $42$, and infinitely many other orders
that were previously unknown. The construction is implemented in our GAP~\cite{GAP4}
package \emph{Prescribed Automorphism Groups}~\cite{PAG}. Examples can be easily
obtained by typing \texttt{Paley3DMat}$(v)$.

The smallest order $v\equiv 2 \pmod{4}$ not covered by Theorem~\ref{maintm} is $v=22$.
A four-dimensional Hadamard matrix of order~$22$ can be constructed by
\cite[Theorem~6.1.4]{YNX10} from a $2$-dimensional Hadamard matrix of order $22^2=484$.
Theorem~\ref{dimpp} then covers all dimensions $n>4$. As far as we know, the existence
of a three-dimensional Hadamard matrix of order~$22$ is an open problem. We tried
constructing examples by prescribing automorphism groups (see~\cite{dLS08}
for the relevant definitions) but we did not succeed.

\end{document}